\documentclass[11pt, a4paper]{article}
\usepackage{}
\usepackage{amsthm}
\usepackage{mathrsfs}
\usepackage{amsmath,amssymb,latexsym,color}
\usepackage[colorlinks,
linkcolor=blue,
anchorcolor=green,
citecolor=magenta
]{hyperref}
\usepackage{graphicx}
\usepackage{tikz}
\usetikzlibrary{calc}
\usepackage{CJK}

\usepackage{authblk}

\long\def\delete#1{}

\usepackage{color}

\definecolor{Blue}{rgb}{0,0,1}
\definecolor{Red}{rgb}{1,0,0}
\definecolor{DarkGreen}{rgb}{0,0.6,0}
\definecolor{DarkYellow}{rgb}{1,1,0.2}
\definecolor{DarkPurple}{rgb}{.6,0,1}

\usepackage{xcolor}
\usepackage[normalem]{ulem}

\usepackage{enumerate}
\usepackage{cleveref}
\crefformat{section}{\S#2#1#3}
\crefformat{subsection}{\S#2#1#3}
\crefformat{subsubsection}{\S#2#1#3}
\crefrangeformat{section}{\S\S#3#1#4 to~#5#2#6}
\crefmultiformat{section}{\S\S#2#1#3}{ and~#2#1#3}{, #2#1#3}{ and~#2#1#3}
\def\ma{\mathcal{A}}
\def\mb{\mathcal{B}}

\def\md{\mathcal{D}}

\def\mf{\mathcal{F}}

\def\mh{\mathcal{H}}
\def\mi{\mathcal{I}}

\def\mt{\mathcal{T}}

\def\mw{\mathcal{W}}

\def\bs{\setminus}

\def\le{\leqslant}
\def\b{\brack}

\numberwithin{equation}{section}

\newtheorem{thm}{Theorem}[section]
\newtheorem{lem}[thm]{Lemma}

\newtheorem{cor}[thm]{Corollary}
\newtheorem{pr1}[thm]{Proposition}

\newtheorem{cl}{Claim}

\newtheorem{con}{Construction}

\begin{document}
	
	\setcounter{page}{1}
	\renewcommand{\thefootnote}{}
	\newcommand{\remark}{\vspace{2ex}\noindent{\bf Remark.\quad}}
	\renewcommand{\abovewithdelims}[2]{%
		\genfrac{[}{]}{0pt}{}{#1}{#2}}

	
	\def\qed{\hfill$\Box$\vspace{11pt}}
	
	\title {\bf $s$-almost $t$-intersecting families for vector spaces}
	
	\author{Lijun Ji\thanks{E-mail: \texttt{jilijun@Suda.edu.cn}}\   \textsuperscript{a}}
	
	\author{Dehai Liu\thanks{E-mail: \texttt{liudehai@mail.bnu.edu.cn}}\   \textsuperscript{b}}
	\author{Kaishun Wang\thanks{ E-mail: \texttt{wangks@bnu.edu.cn}}\ \textsuperscript{b}}
	\author{Tian Yao\thanks{Corresponding author. E-mail: \texttt{tyao@hist.edu.cn}}\ \textsuperscript{c}}
	\author{Shuhui Yu\thanks{ E-mail: \texttt{yushuhui\_suda@163.com}}\ \textsuperscript{d}}
	
	\affil{ \textsuperscript{a} Department of Mathematics,
		Soochow University, Suzhou 215006,  China}
	
	\affil{ \textsuperscript{b} Laboratory of Mathematics and Complex Systems (Ministry of Education), School of
		Mathematical Sciences, Beijing Normal University, Beijing 100875, China}
	
	\affil{ \textsuperscript{c} School of Mathematical Sciences, Henan Institute of Science and Technology, Xinxiang 453003, China}
	
	\affil{ \textsuperscript{d} Department of Mathematics and Physics, Suzhou Vocational University, Suzhou	215000, China}
	
	\date{}
	
	\openup 0.5\jot
	\maketitle

	\begin{abstract}
		
		Let $V$ be a finite dimensional vector space over a finite field, and $\mathcal{F}$ a family consisting of $k$-subspaces of $V$. The family $\mathcal{F}$ is 
		called $t$-intersecting if $\dim(F_{1}\cap F_{2})\geq t$ for any $F_{1}, F_{2}\in \mathcal{F}$. We say $\mathcal{F}$ is $s$-almost $t$-intersecting if for each $F\in \mathcal{F}$ there are at most $s$ members $F^{\prime}$ of $\mathcal{F}$ such that $\dim(F\cap F^{\prime})<t$.
		In this paper, we prove that $s$-almost $t$-intersecting families with maximum size are $t$-intersecting. 
		We also consider $s$-almost $t$-intersecting families which are not $t$-intersecting, and characterize such families with maximum size for $(s,t)\neq(1,1)$. 
		The result for $1$-almost $1$-intersecting families provided by Shan and Zhou is generalized.
		
		\vspace{2mm}
		
		\
		
		
	\end{abstract}
	\section{Introduction}
	Let $\mf$ be a family consisting of $k$-subsets of an $n$-set. It is called  \textit{$t$-intersecting} if $\left| F_{1}\cap F_{2}\right|\geq t$ for any $F_{1}, F_{2}\in \mf$. 
	The famous Erd\H{o}s-Ko-Rado theorem \cite{MR0140419} gives the structure of maximum-sized $t$-intersecting families.  
	We refer readers to  \cite{ 2406295, MR0140419, 2407161, MR0519277, 2407162, MR0771733} for extensive  results on $t$-intersecting families.
	
	We say that $\mf $ is  \textit{$s$-almost $t$-intersecting} if for each $F\in \mf$ there are at most $s$ members  $F^{\prime}$ of $\mf$ with $\left| F\cap F^{\prime}\right|< t$. Clearly, a $t$-intersecting family is $s$-almost $t$-intersecting. 
	Gerbner et al. \cite{2406095}  proved that,  when  $n$ is sufficiently large, each member of a maximum-sized $s$-almost $1$-intersecting family contains a fixed element. 
	In \cite{2406096}, Frankl and Kupavskii characterized the structure of maximum $1$-almost $1$-intersecting families under the condition that they are not $1$-intersecting.   There are also some other results concerning families without $t$-intersection property \cite{2406092,2406094,2406093}.

	Intersection problems are  studied on some other mathematical objects, for example, vector spaces. Let $V$ be an $n$-dimensional vector space over the finite field $\mathbb{F}_{q}$. Write  the family of all
	$k$-subspaces of $V$ as ${V\b k}$. A family $\mf\subseteq {V\b k}$ is said to be \textit{$t$-intersecting} if $\dim(F_{1}\cap F_{2})\geq t$ for any $F_{1},F_{2}\in \mf$.
	Some classical results were proved for vector spaces \cite{ 2406292, 2406293, MR0867648,MR0382015,MR2231096, 2406294,2405263}. 
	
	For  a family $\mf\subseteq {V\b k}$ and  a subspace $S$ of $V$, define
	$$\md_{\mf}(S;t)=\left\{ F\in\mf: \dim(F\cap S)<t\right\}.$$ 
	A family $\mf\subseteq {V\b k }$ is  also called \textit{$s$-almost $t$-intersecting} if $\left|\md_{\mf}(F;t)\right|\leq s$ for any $F\in\mf$.  Shan and  Zhou \cite{2406097} characterized $1$-almost $1$-intersecting families with maximum size. Our  first result provides the structure of maximum-sized $s$-almost $t$-intersecting families   for general $s$ and $t$.

	\begin{thm}\label{6}
		Let $n$, $k$, $t$ and $s$ be positive integers with $k\geq t+1$ and  $n\geq 2k+s+\delta_{2,q}(1-\delta_{k,t+1})$. If $\mathcal{F}\subseteq {V\brack k}$ is a maximum-sized $s$-almost $t$-intersecting family, then there exists $E\in {V\brack t}$ such that
		$\mathcal{F}=\lbrace F\in{V\brack k}: E\subseteq F \rbrace$.
	\end{thm}
	
	Theorem  \ref{6} states that maximum-sized $s$-almost $t$-intersecting families  are $t$-intersecting.  
	It is natural to consider families without $t$-intersection property.	
	To present our another  result,  we introduce some notations and a family which is $s$-almost $t$-intersecting but not $t$-intersecting.
	
	Let $n$, $k$, $t$ and $s$ be positive integers with $k\geq t+1$ and $n\geq 2k-t+s$. Suppose $M\in{V\b k}$ and $E\in{M\b t}$.  
	Denote
	$\mh_{1}(M,E)=\left\{F\in{V\b k}: F\cap M=E\right\}$ and
	$\mh_{2}(M,E)=\left\{F\in{M\b k}: E\nsubseteq F\right\}.$
	\begin{con}\label{2406191} 
	Let $n$, $k$, $t$, $s$, $M$ and $E$ be as above. Consider
	$$\left\{F\in{V\b k}: E\subseteq F,\ \dim(F\cap M)\geq t+1\right\}\cup \ma \cup \mb,$$
	where	$\ma$ is an $s$-subset of $\mh_{1}(M,E)$ and $\mb$ is a $\min\left\{s, q^{k-t+1}{t\b 1}\right\}$-subset of $\mh_{2}(M,E)$. 
	\end{con}
		Shan and Zhou \cite{2406097} also studied $s$-almost $t$-intersecting families 
	which are not $t$-intersecting for $t=s=1$.  Our second result focuses on the case $(t,s)\neq (1,1)$.
	\begin{thm}\label{2407163}
		Let $n$, $k$,  $t$ and $s$ be positive integers with  $(t,s)\neq (1,1)$, $k\geq t+2$ and  $n\geq 2k+t+s-1+\delta_{2,q}$. Suppose $\mf\subseteq {V\b k}$ is an $s$-almost $t$-intersecting family which is not $t$-intersecting.  If $\mf$ has  maximum size, then there exist $M\in {V\b k+1}$ and $E\in {M\b t}$ such that 
		$$\mf=\left\{F\in{V\b k}: E\subseteq F,\ \dim(F\cap M)\geq t+1\right\}\cup \ma \cup \mb,$$
		where	$\ma$ is an $s$-subset of $\mh_{1}(M,E)$ and $\mb$ is a $\min\left\{s, q^{k-t+1}{t\b 1}\right\}$-subset of $\mh_{2}(M,E)$.
	\end{thm}

	The rest of this paper is organized as follows. In Section \ref{2407062}, we present some inequalities   needed in this paper.  Then we prove Theorem \ref{6} in Section \ref{2407063}.
	In Section \ref{2407066},  Theorem \ref{2407163} is proved. Finally in  Section \ref{2407065}, we show Proposition \ref{2407081} which is used in Section \ref{2407066}.

	\section{Inequalities concerning Gaussian binomial coefficients}\label{2407062}
	
	We begin with some lemmas about the Gaussian binomial coefficient ${n\b k}$. The first one can be easily checked. 
	
	\begin{lem}\label{1}
		Let $n$ and $k$ be positive integers with $k<n$. Then the following hold.
		\begin{enumerate}[\normalfont(i)]
			\item $q^{k-1}\geq k$.
			\item ${k\b 1}\leq \frac{3}{2}q^{k-1}$ if $q\geq 3$.
			\item $q^{n-k}<\frac{q^{n}-1}{q^{k}-1}<q^{n-k+1}$.
			\item  $q^{k(n-k)}<{n\brack k}<q^{k(n-k+1)}$.	
			\item  If $t$ is a positive integer with $k\geq t$ and $n\geq 2k-t+1$, then
			\begin{equation*}
				{k-t+1\brack 1}^{j-i}{n-j\brack k-j}\leq{n-i\brack k-i}
			\end{equation*}
			for any positive integers $i$ and $j$ with $i\leq j$.
		\end{enumerate}
	\end{lem}

	\begin{lem}\label{2405261}{\textnormal{(\cite[Lemma 2.1]{2403213}) }}
		Let $n$ and $k$ be non-negative integers with $k\leq n$. Then the following hold.
		\begin{enumerate}[\normalfont(i)]
			\item If $q\geq 2$, then
			$${n\brack k}\leq  \dfrac{7}{2}q^{k(n-k)}.$$	
			\item If $q\geq3$, then $${n\brack k}\leq  2q^{k(n-k)}.$$	
		\end{enumerate}
	\end{lem}
	
		For positive integers $n$, $k$, $t$, $s$ and $x$, write
	\begin{equation}\label{2407186}
		\begin{aligned}
			f(n,k,t,s,x)=&\ {x\brack t}{k-t+1\brack 1}^{x-t}{n-x\brack k-x}+s{x\brack t}\sum_{i=0}^{x-t-1}{k-t+1\brack 1}^{i}.\\
			g(n,k,t,s, x)= &\ {x-t\brack 1}{n-t-1\brack k-t-1}+q^{x-t}{k-x+1\brack 1}{k-t+1\brack 1}{n-t-2\brack k-t-2}\\
			&+s q^{x-t}{k-x+1\brack 1}+q^{x -t+1}{t\brack 1}{k-t\brack 1}{n-x \brack k-x}+2s.
		\end{aligned}
	\end{equation}
	\begin{lem}\label{5}
		Let $n$, $k$, $t$ and $s$ be positive integers with  $k\geq t+2$ and $n\geq 2k+s+\delta_{2,q}$. Then $f(n,k,t,s,x)> f(n,k,t,s,x+1)$  for $x\in \lbrace t+1,\ldots, k-1 \rbrace$.
	\end{lem}
	\begin{proof} 
		Note that $$f(n,k,t,s,x)\geq {x\b t}{k-t+1\brack 1}^{x-t}{n-x\brack k-x}.$$ Then for any $x\in \lbrace t+1,t+2,\ldots, k-1 \rbrace$, we have  
		\begin{equation*}
			\begin{aligned}
				\frac{f(n,k,t,s, x+1)}{f(n,k,t,s, x)}\leq  \frac{(q^{x+1}-1)(q^{k-x}-1){k-t+1\b1}}{(q^{x+1-t}-1)(q^{n-x}-1)}+\frac{s(q^{x+1}-1)\sum_{i=0}^{x-t}{k-t+1\b 1}^{i}}{(q^{x+1-t}-1){k-t+1\b 1}^{x-t}{n-x\b k-x}}.
			\end{aligned}
		\end{equation*}
		
		Suppose $q\geq 3$. It is routine to check that
		$$ \frac{8}{9}q^{x+1-t}\leq q^{x+1-t}-1,\quad \sum_{i=0}^{x-t}{k-t+1\b 1}^{i}\leq \frac{1}{12}{k-t+1\b 1}^{x-t+1}.$$
		By Lemma \ref{1},  we obtain
		\begin{equation*}\label{2406081}
			\frac{f(n,k,t,s,x+1)}{f(n,k,t,s,x)}\leq\frac{27q^{x+1}\cdot q^{k-t}}{16q^{x+1-t}\cdot q^{n-k}}+\frac{9s q^{x+1}\cdot q^{k-t}}{64q^{x+1-t}\cdot q^{(k-x)(n-k)}}\leq \frac{9}{16}+\frac{3}{64}<1.
		\end{equation*}
		
		Suppose $q=2$. 
		We have
		$$\frac{3}{4}q^{x+1-t}\leq q^{x+1-t}-1,\quad \sum_{i=0}^{x-t}{k-t+1\b 1}^{i}\leq \frac{1}{6}{k-t+1\b 1}^{x-t+1}.$$
		By Lemma \ref{1}, 
		we get 
		\begin{equation*}\label{2406082}
			\frac{f(n,k,t,,s,x+1)}{f(n,k,t,s, x)}\leq\frac{4q^{x+1}\cdot q^{k-t+1}}{3q^{x+1-t}\cdot q^{n-k}}+\frac{2s q^{x+1}\cdot q^{k-t+1}}{9q^{x+1-t}\cdot q^{(k-x)(n-k)}}\leq \frac{2}{3}+\frac{1}{9}<1.
		\end{equation*}
		
		Hence, the desired result follows.
	\end{proof}
	
	\begin{lem}\label{2407176}
		Let $n$, $k$, $t$ and $s$ be positive integers with  $k\geq t+3$ and $n\geq 2k+s+\delta_{2,q}$. Then $g(n,k,t,s,x)< g(n,k,t,s,x+1)$  for $x\in \lbrace t+2,\ldots, k-1 \rbrace$.
	\end{lem}
	\begin{proof}
		By Lemma \ref{1}, we obtain 
		\begin{equation*}
			\begin{aligned}
				\frac{q^{n-t-1}-1}{q^{k-t-1}-1}-{k-t+1\b 1} -q {t\b 1}{k-t\b 1}
				\geq q^{n-k}-\frac{3}{2}q^{k-t}-\frac{9}{4}q^{k-1}
				\geq q^{k+1}-\frac{15}{4}q^{k-1}\geq 1
			\end{aligned}
		\end{equation*}
		for $q\geq 3$, and 
		\begin{equation*}
			\begin{aligned}
				\frac{q^{n-t-1}-1}{q^{k-t-1}-1}-{k-t+1\b 1} -q {t\b 1}{k-t\b 1}
				\geq  q^{k+s +1}-q^{k-t+1}-q^{k+1}\geq 1
			\end{aligned}
		\end{equation*}
		for $q=2$.
		Therefore,  we have 
		\begin{equation*}
			\begin{aligned}
				&\ g(n,k,t,x+1)-g(n,k,t,x)\\
				=&\ q^{x-t}\frac{q^{n-t-1}-1}{q^{k-t-1}-1}{n-t-2\b k-t-2}-q^{x-t}{k-t+1\b1}{n-t-2\b k-t-2}-q^{x-t+1}{t\b 1}{k-t\b 1}{n-x\b k-x}\\
				&\ -s q^{x-t}
				+q^{x-t+2}{t\b 1}{k-t\b 1}{n-x-1\b k-x-1}\\
				\geq &\ q^{x-t}{n-t-2\b k-t-2} \left( \frac{q^{n-t-1}-1}{q^{k-t-1}-1}-{k-t+1\b1}\right)  -q^{x-t+1}{t\b 1}{k-t\b 1}{n-t-2\b k-t-2}
				 -s q^{x-t}\\
				=&\ q^{x-t}{n-t-2\b k-t-2}\left(\frac{q^{n-t-1}-1}{q^{k-t-1}-1}-{k-t+1\b 1} -q {t\b 1}{k-t\b 1}\right)-s q^{x-t}\\
				\geq &\ q^{x-t}\cdot q^{n-k}-s q^{x-t}>0,
			\end{aligned}
		\end{equation*}
		as desired.
	\end{proof}

	\begin{lem}\textnormal{(\cite{04162})}\label{2404234}
		Let $n$, $a$, $b$, $b^{\prime}$, $i$ and $i^{\prime}$ be non-negative integers with $0\leq i^{\prime}\leq i\leq a$ and $0\leq b^{\prime}-i^{\prime}\leq b-i\leq n-a$. Suppose $A\in {V\brack a}$ and  $C\in {V\brack b^{\prime}}$ with $\dim(C\cap A)=i^{\prime}$. Then 
		$$\left| \left\{ B\in {V\brack b}: C\subseteq B,\ \dim(B\cap A)=i    \right\}\right| =q^{(a-i)((b-i)-(b^{\prime}-i^{\prime}))}{n-a-(b^{\prime}-i^{\prime})\brack (b-i)-(b^{\prime}-i^{\prime})} {a-i^{\prime}\brack i-i^{\prime}}.$$  
	\end{lem}

	Let  $n$, $k$, $t$ and $s$ be positive integers with $k\geq t+1$ and $n\geq 2k-t+s$.	Assume $M\in {V\b k+1}$ and $E\in {M\brack t}$. Suppose that $\ma$ is an $s$-subset of $\mh_{1}(M,E)$ and $\mb$ is a $\min\left\{ s, q^{k-t+1}{t\b 1}\right\}$-subset of $\mh_{2}(M,E)$.
	By Lemma \ref{2404234}, it is routine to check 
	\begin{equation*}
		\begin{aligned}
			h(n,k,t, s):&=\left|\left\{F\in{V\b k}: E\subseteq F,\ \dim(F\cap M)\geq t+1\right\}\cup \ma \cup\mb\right|\\
			&={n-t\brack k-t}-q^{(k-t+1)(k-t)}{n-k-1\brack k-t}+s+ \min\left\{ s, q^{k-t+1}{t\b 1}\right\}.
		\end{aligned}
	\end{equation*}

	\begin{lem}\label{2405262}
		Let $n$, $k$, $t$ and $s$ be positive integers with $(t,s)\neq  (1,1)$, $k\geq t+2$ and $n\geq 2k+t+s-1+\delta_{2,q}$. Then the following hold.
		\begin{enumerate}[\normalfont(i)]
			\item If $q\geq 2$, then
			$$h(n,k,t,s)>\frac{23}{24}{k-t+1\b1}{n-t-1\b k-t-1}.$$
			\item If $q\geq 3$, then
			$$h(n,k,t,s)>\frac{71}{72}{k-t+1\b1}{n-t-1\b k-t-1}.$$
		\end{enumerate}	
	\end{lem}
	\begin{proof}
		From   \cite[Lemma 2.4 and (2.11)]{2405263}, we obtain 
		$$h(n,k,t,s)> {k-t+1\b 1}{n-t-1\b k-t-1}-q{k-t+1\b 2}{n-t-2\b k-t-2}.$$
		This together with Lemma \ref{1}  yields
		$$\frac{h(n,k,t,s)}{{k-t+1\b1}{n-t-1\b k-t-1}}>1-\frac{1}{(q^{2}-1)q^{n-2k+t-1}},$$
		which implies the  desired results.
	\end{proof}

	\section{Proof of Theorem \ref{6}}\label{2407063}
	
	Let $\mf\subseteq{V\b k}$. A subspace $W$  of $V$ is called a \textit{$t$-cover} of $\mf$ if $\dim(W\cap F)\geq t$ for any $F\in \mf$.
	Define the \textit{$t$-covering number} $\tau_{t}(\mf)$ of $\mf$ as the minimum dimension of a $t$-cover of $\mf$.
	Before proving Theorem \ref{6},  we  show some  properties of $s$-almost $t$-intersecting families.

	\begin{lem}\label{2404224}
	Let $n$, $k$, $t$ and $s$ be positive integers with $k\geq t+1$ and $n\geq 2k$. Suppose that $\mathcal{F}\subseteq {V\brack k}$ is $s$-almost $t$-intersecting  and $S\in {V\b s}$. If  there exists $F\in\mf$ such that 
	$\dim(F\cap S)=m< t$, then 
	$$\left| \mathcal{F}_{S}\right|\leq {k-t+1\brack 1}^{t-m}\left| \mathcal{F}_{R} \right|+s$$
	for some $R\in {V\brack s+t-m}$ with $S\subseteq R$.
	\end{lem}
	\begin{proof}
	Let $\mi=\mf\bs\md_{\mf}(F;t)$. 
	Notice that $F$ is a $t$-cover of $\mi$. By \cite[Lemma 2.4]{1}, there exists an $(s+t-m)$-subspace $R$ with $S\subseteq R$ such that 
	\begin{equation}\label{2407165}
		\left| \mathcal{I}_{S}\right|\leq {k-m\brack t-m}\left| \mathcal{I}_{R} \right|\leq {k-m\brack t-m}\left| \mathcal{F}_{R} \right|.
	\end{equation}
	Since $\mf$ is $s$-almost $t$-intersecting, we have $\left|\mf\bs \mi\right|\leq s$.  
	The fact $\mi\subseteq \mf$ implies
	$\left|\mf_{S}\right|\leq \left|\mi_{S}\right|+s$. 
	It is simple to verify that $$\frac{(q^{k-m}-1)(q^{k-m-1}-1)\cdots (q^{k-t+1}-1)}{(q^{t-m}-1)(q^{t-m-1}-1)\cdots(q-1)}\leq \left( \frac{q^{k-t+1}-1}{q-1}\right)^{t-m},$$ i.e., ${k-m\b t-m}\leq {k-t+1\b 1}^{t-m}$.
	Then from (\ref{2407165}) we get the desired result.	
	\end{proof}
	
	\begin{pr1}\label{4}
		Let $n$, $k$, $t$ and $s$ be positive integers with $k\geq t+1$ and $n\geq 2k$. Suppose that $\mathcal{F}\subseteq {V\brack k}$ is $s$-almost $t$-intersecting, and $\mb\subseteq \mf$. If  $A$ is a $t$-cover of $\mf \bs\mb$ with $\dim(A)<\tau_{t}(\mathcal{F})\le k$, then there exists a subspace $A_{u}$ of $V$ satisfying $A\subseteq A_{u}$, $\tau_{t}(\mf)\leq \dim(A_{u})$ and 
		$$\left| \mb_{A}\right|\leq {k-t+1\b 1}^{\dim(A_{u})-\dim(A)}\left|\mb_{A_{u}}\right|+s\sum_{i=0}^{\tau_{t}(\mf)-\dim(A)-1}{k-t+1\b 1}^{i}.$$
	\end{pr1}
	\begin{proof}
		
		Since $\dim(A)<\tau_{t}(\mathcal{F})$ and $\dim(A\cap F)\geq t$ for each $F\in \mf\bs \mb$, 
		there exists  $F_{1}\in \mb$ such that $\dim(F_{1}\cap A)<t$. 
		Note that $\mb$ is also $s$-almost $t$-intersecting. By Lemma \ref{2404224}, there exists a subspace $A_{2}$ with $A\subsetneq A_{2}$ and 
		$$\left|\mb_{A}\right|\leq {k-t+1\b 1}^{\dim(A_{2})-\dim(A_{})}\left|\mb_{A_{2}}\right|+s.$$
		If $\dim(A_{2})\geq \tau_{t}(\mf)$, then the desired result follows. If $\dim(A_{2})<\tau_{t}(\mathcal{F})$, then  apply Lemma \ref{2404224} on $A_{2}$.  Using Lemma \ref{2404224} repeatedly,
		we finally conclude that there exists an ascending chain of subspaces $A=A_{1}\subsetneq A_{2}\subsetneq\ldots \subsetneq A_{u}$  such that 
		$\dim(A_{u-1})<\tau_{t}(\mathcal{F})\leq \dim(A_{u})$ and 
		$$\left|\mb_{A_{i}}\right|\leq {k-t+1\brack 1}^{\dim(A_{i+1})-\dim(A_{i})}\left| \mb_{A_{i+1}}\right|+s, \quad 1\leq i\leq u-1.$$
		Hence, we derive 
		\begin{equation*}
			\left| \mb_{A}\right|\leq {k-t+1\brack 1}^{\dim(A_{u})-\dim(A)}\left| \mb_{A_{u}}\right|+s\cdot\sum_{i=1}^{u-1}{k-t+1\brack 1}^{\dim(A_{i})-\dim(A)}.
		\end{equation*}
		This together with $\dim(A_{u-1})\leq \tau_{t}(\mf)-1$ yields the desired result. 
	\end{proof}
	
	Setting $\mb=\mf$ in Proposition \ref{4}, by Lemma \ref{1} (v), we get an upper bound on $\left|\mf_{A}\right|$.
	
	\begin{cor}\label{2407166}
		Let $n$, $k$ and $t$ be positive integers with $k\geq t+1$ and $n\geq 2k$. 
		Suppose that $\mathcal{F}\subseteq {V\brack k}$ is $s$-almost $t$-intersecting. If $A$ is a subspace of $V$ with $\dim(A)<\tau_{t}(\mathcal{F})\le k$, then
		$$\left| \mathcal{F}_{A} \right|\leq {k-t+1\brack 1}^{\tau_{t}(\mathcal{F})-\dim(A)}{n-\tau_{t}(\mathcal{F})\brack k-\tau_{t}(\mathcal{F})}+s\sum_{i=0}^{\tau_{t}(\mathcal{F})-\dim(A)-1}{k-t+1\brack 1}^{i}.$$
	\end{cor}

	 \begin{lem}\label{2407037}
		Let $n$, $k$, $t$ and $s$ be positive integers with $k\geq t+1$ and $n\geq 2k$. If $\mf\subseteq {V\b k}$ is an $s$-almost $t$-intersecting family with $\tau_t(\mf)\geq k+1$, then $|\mf|\leq s\binom{2k-2t+2}{k-t+1}$.
	\end{lem}
	\begin{proof}
		Let $B_{1}\in \mf$. Since $\tau_t(\mf)\geq k+1$, we have  $\dim(B_{1}\cap A_{1})<t$ for some $A_{1}\in \mf$.
		If $\md_{\mf}(A_{1};t)=\mf$, then there is nothing to do. If $\md_{\mf}(A_{1};t)\subsetneq\mf$, then choose $B_{2}\in \mf\backslash \md_{\mf}(A_{1};t)$. There exists $A_{2}\in \mf$ such that $\dim(A_{2}\cap B_{2})<t$ due to $\tau_t(\mf)\geq k+1$. Repeating the process above, we finally conclude that there exist two series subspaces $A_{1}, A_{2}, \ldots, A_{m}$ and $B_{1}, B_{2},\ldots, B_{m}$ such that:
		\begin{enumerate}[\normalfont(i)]
			\item $\dim(A_{i}\cap B_{i})<t$ for any $i$.
			\item $\dim(A_{i}\cap B_{j})\geq t$ for any $i<j$.
			\item $\bigcup _{i=1}^{m}\md_{\mf}(A_{i};t)=\mf$.
		\end{enumerate}
	    By (i), (ii) and  \cite[Theorem 5]{2406291}, we know $m\leq \binom{2k-2t+2}{k-t+1}$. This together with (iii) and $\left|\md_{\mf}(F;t)\right|\leq s$ for any $F\in \mf$ yields the desired result.
	\end{proof}
	
	\begin{proof}[\bf Proof of Theorem \ref{6}] Note that $\tau_{t}(\mf)\geq t$. If $\tau_{t}(\mf)=t$, then the desired result follows from the  maximality of $\mf$. Next we consider the case $\tau_{t}(\mf)\geq t+1$. 
		To finish our proof, it is sufficient to show 
		\begin{equation*}\label{2406151}
			\left|\mf\right|< {n-t\b k-t}.
		\end{equation*}

		\medskip
		\noindent \textbf{Case 1.} $t+1\leq \tau_{t}(\mathcal{F})\leq k$. 
		\medskip
		
		Let $T$ be a  $t$-cover of $\mf$ with dimension $\tau_{t}(\mf)$. Then 
		$\mathcal{F}=\bigcup_{H\in {T\brack t}}\mathcal{F}_{H}$.
		It follows   from Lemma \ref{5} and Corollary \ref{2407166} that 
		\begin{equation}\label{2407167}
			\left| \mathcal{F}\right|\leq {t+1\b 1}{k-t+1\b 1}{n-t-1\b k-t-1}+s{t+1\b 1}.
		\end{equation}

		\medskip
		\noindent \textbf{Case 1.1.} $q\geq 3$. 
		\medskip
		
		By (\ref{2407167}), Lemma \ref{1} and $n\geq 2k+s$, we have 
		\begin{equation*}
			\begin{aligned}
				\frac{\left|\mf\right|}{{n-t\b k-t}}\leq \frac{9}{4q^{n-2k}}+\frac{3s}{2q^{n-k-t}}\leq 
				\frac{9}{4q^{s}}+\frac{3}{2q^{2}}\cdot\frac{s}{q^{s-1}}\leq  \frac{3}{4}+\frac{1}{6}<1.
			\end{aligned}
		\end{equation*}

		\medskip
		\noindent \textbf{Case 1.2.} $q= 2$. 
		\medskip
		
		If $k=t+1$, then  by (\ref{2407167}), Lemma \ref{1}  and $n\geq 2k+s$, we get 
		\begin{equation*}\label{2405287}
			\begin{aligned}
				\frac{\left|\mf\right|}{{n-t\b k-t}}\leq \frac{3(q^{t+1}-1)}{q^{n-t}-1}+\frac{s(q^{t+1}-1)}{q^{n-t}-1}<\frac{3}{q^{n-2t-1}}+\frac{s}{q^{n-2t-1}}\leq \frac{3}{q^{t+s}}+\frac{s}{q^{t+1}\cdot q^{s-1}}\leq 1.
			\end{aligned}
		\end{equation*}
		
		If $k\geq t+2$, then   by (\ref{2407167}), Lemma \ref{1}  and $n\geq 2k+s +1$, we obtain
		\begin{equation*}\label{2405286}
			\begin{aligned}
				\frac{\left|\mf\right|}{{n-t\b k-t}}\leq  \frac{(q^{t+1}-1)(q^{k-t+1}-1)}{q^{n-k}}+\frac{s(q^{t+1}-1)}{q^{n-k}}
				\leq  \frac{q^{t+1}(q^{k-t+1}-1)}{q^{k+2}}+\frac{q^{t+1}-1}{q^{k+2}}\cdot\frac{s}{q^{s-1}}
				<1.
			\end{aligned}
		\end{equation*}

		
		\medskip
		\noindent \textbf{Case 2.} $\tau_{t}(\mathcal{F})\geq k+1$. 
		\medskip
		
		Note that $n\geq 2k+s$ and $q^{z}\geq z+1$ for non-negative integer $z$.
		By Lemma \ref{1}, we have
		\begin{equation*}
			\begin{aligned}
				\frac{s\binom{2k+2-2t}{k+1-t}}{{n-t\b k-t}}=\frac{2s(2+\frac{1}{k-t})\binom{2k-2t}{k-t-1}}{\frac{q^{n-t}-1}{q^{k-t}-1}{n-t-1\b k-t-1}}
				\leq \frac{6s}{q^{n-k}}\cdot \left(\frac{k-t+2}{q^{n-k}}\right)^{k-t-1}
				\leq  \frac{3}{4}\cdot  \left(\frac{k-t+2}{n-k+1}\right)^{k-t-1}\leq \frac{3}{4}.
			\end{aligned}
		\end{equation*}
		This together with Lemma \ref{2407037} yields the desired result.
	\end{proof}

	\section{Proof of Theorem \ref{2407163}} \label{2407066}

	Before starting our proof, we give some auxiliary results.

		\begin{lem}\label{2405283}
		Let  $x_{1}$, $x_{2}$ and $x_{3}$ be non-negative integers. Suppose that $U_{1}, U_{2}$ and $U_{3}$ are subspaces of $V$. If $\dim(U_{3}\cap U_{1})\geq x_{1}$, $\dim(U_{3}\cap U_{2})\geq x_{2}$, $\dim(U_{1}\cap U_{2})\leq x_{3}$ and $x_{1} +x_{2}-x_{3}=\dim(U_{3})$, then $\dim(U_{3}\cap U_{1})= x_{1}$, $\dim(U_{3}\cap U_{2})=x_{2}$, $\dim(U_{1}\cap U_{2})=x_{3}$, $U_{1}\cap U_{2}\subseteq U_{3}$ and $U_{3}=U_{3}\cap U_{1}+U_{3}\cap U_{2}$.
	\end{lem}
	\begin{proof}
		By dimension formula for vector spaces, we have 
		\begin{equation*}
			\begin{aligned}
				\dim(U_{3})&\geq \dim(U_{3}\cap U_{1}+U_{3}\cap U_{2})\\
				&=\dim(U_{3}\cap U_{1})+\dim(U_{3}\cap U_{2})-\dim(U_{3}\cap U_{1}\cap U_{2})\\
				&\geq 	\dim(U_{3}\cap U_{1})+\dim(U_{3}\cap U_{2})-\dim( U_{1}\cap U_{2})\\
				&\geq x_{1}+x_{2}-x_{3}\\
				&=\dim(U_{3}).	 
			\end{aligned}
		\end{equation*}
		It follows that three inequalities above are all equalities. 
		Then the desired results hold.
	\end{proof}

	\begin{lem}\label{2407022}
		Let $n$, $k$, $t$ and $s$ be positive integers with $k\geq t+1$ and $n\geq 2k$. Suppose  $\mathcal{F}\subseteq {V\brack k}$ is an $s$-almost $t$-intersecting family with $\tau_{t}(\mf)<k$, and has the unique  $\tau_{t}(\mf)$-dimensional $t$-cover $T$. If $A$ is a subspace of $V$ such that $\dim(A)<\tau_{t}(\mathcal{F})$ and $\dim(A\cap T)\geq t$, then 
		\begin{equation*}\label{2407023}
			\left| \left( \mf\bs \mf_{T}\right)_{A}\right|\leq {k-t+1\b 1}^{\tau_{t}(\mf)-\dim(A)+1}{n-\tau_{t}(\mf)-1\b k-\tau_{t}(\mf)-1}+s\sum_{i=0}^{\tau_{t}(\mf)-\dim(A)}{k-t+1\b 1}^{i}. 
		\end{equation*}
	\end{lem}
	\begin{proof}
		Let $\mb=\mf\bs \mf_{T}$. Observe that $A$ is a $t$-cover of $\mf \bs \mb$. By Proposition \ref{4}, there is
		a subspace $A_{u}$ of $V$ such that $A\subseteq A_{u}$, $\tau_{t}(\mf)\leq \dim(A_{u})$ and 
		\begin{equation}\label{2406155}
			\left| (\mf\bs \mf_{T})_{A}\right|\leq {k-t+1\b 1}^{\dim(A_{u})-\dim(A)}\left|(\mf\bs \mf_{T})_{A_{u}}\right|+s\sum_{i=0}^{\tau_{t}(\mf)-\dim(A)-1}{k-t+1\b 1}^{i}.
		\end{equation}
		If $A_{u}=T$, then $\left| (\mf\bs \mf_{T})_{A_{u}}\right|=0$, which together with (\ref{2406155}) implies the desired result. 
		In the following, assume that $A_{u}\neq T$.
		
		Suppose $\dim(A_{u})\geq \tau_{t}(\mf)+1$. By Lemma \ref{1} (v), we have 
		$${k-t+1\brack 1}^{\dim(A_{u})-\dim(A)}{n-\dim(A_{u})\b k-\dim(A_{u})}\leq {k-t+1\b 1}^{\tau_{t}(\mf)-\dim(A)+1}{n-\tau_{t}(\mf)-1\b k-\tau_{t}(\mf)-1}.$$
		This combining with
		$\left| (\mf\bs \mf_{T})_{A_{u}}\right|\leq {n-\dim(A_{u})\b k-\dim(A_{u})}$ and (\ref{2406155}) implies the desired result.

		Suppose $\dim(A_{u})=\tau_{t}(\mf)$. Then $A_{u}$ is not a $t$-cover of $\mf$ since $A_{u}\neq T$. Note that $\dim(A_{u}\cap T)\geq t$. 
		We further conclude  $\dim(A_{u}\cap F_{u})<t$ for some $F_{u}\in \mf\bs \mf_{T}$. By Lemma \ref{1} (v) and Lemma \ref{2404224}, there exists a subspace $A_{u+1}$ such that $\dim(A_{u})< \dim(A_{u+1})$ and 
		\begin{equation*}
			\begin{aligned}
				\left| (\mf\bs \mf_{T})_{A_{u}}\right|\leq &{k-t+1\b 1}^{\dim(A_{u+1})-\dim(A_{u})}\left| (\mf\bs \mf_{T})_{A_{u+1}}\right|+s\\
				\leq & {k-t+1\b 1}^{\tau_{t}(\mf)-\dim(A_{u})+1}{n-\tau_{t}(\mf)-1\b k-\tau_{t}(\mf)-1}+s.
			\end{aligned}
		\end{equation*}
		The desired result follows from (\ref{2406155}).
	\end{proof}
	
	\begin{lem}\label{2405122}
		Let $n$, $k$, $t$ and $s$ be positive integers with $k\geq t+1$ and $n\geq 2k+s-1$. Suppose that $\mf\subseteq {V\b k}$ is a maximal $s$-almost $t$-intersecting family with $\tau_{t}(\mf)\leq k$. If $\mt$ is the set of all $t$-covers of $\mf$ with dimension $\tau_{t}(\mf)$, then $\mt$ is $t$-intersecting.
	\end{lem}
	\begin{proof}
		It is sufficient  to show $\dim(T_{1}\cap T_{2})\geq t$ for any $T_{1}, T_{2}\in \mt$. 
		If $\tau_{t}(\mf)=k$, then $T_{2}\in \mt\subseteq \mf$ by the maximality of $\mf$. Since $T_{1}$ is a $t$-cover of $\mf$, we have $\dim(T_{1}\cap T_{2})\geq t$, as desired. 
		In the following,  assume that $\tau_{t}(\mf)<k$.

		Suppose for contradiction that $\dim(T_{1}\cap T_{2})\leq t-1$.   
		Then by Lemma \ref{2404234}, there exists a  $k$-subspace $W_{1}$ such that $W_{1}\cap (T_{1}+T_{2})=T_{1}$. 
		Set
		$$\mathcal{W}=\left\{ W\in {V\brack k}: W\cap (W_{1}+T_{2})=T_{2}\right\}.$$
		Since $t\leq \tau_{t}(\mf)\leq \dim(W_{1}+T_{2})$ and $0\leq k-\tau_{t}(\mf)\leq n-\dim(W_{1}+T_{2})$, we get
		\begin{equation*}
			\begin{aligned}
				\left| \mathcal{W}  \right|= &\ q^{(\dim(W_{1}+T_{2})-\tau_{t}(\mf))(k-\tau_{t}(\mf))}{n-\dim(W_{1}+T_{2})\brack k-\tau_{t}(\mf)}\\
				\geq &\ q^{(k-\tau_{t}(\mf))^{2}}\cdot q^{(k-\tau_{t}(\mf))(n-\dim(W_{1}+T_{2})-k+\tau_{t}(\mf))}\\
				\geq &\ q\cdot q^{s-1}> s
			\end{aligned}
		\end{equation*}
		from Lemma \ref{2404234}.	For each $W\in\mathcal{W}$, we have
		\begin{equation*}
			\begin{aligned}
				W\cap W_{1}=W\cap W_{1}\cap (W_{1}+T_{2})
				=T_{2}\cap W_{1}=(T_{1}+T_{2})\cap T_{2}\cap W_{1}
				=T_{1}\cap T_{2},
			\end{aligned}
		\end{equation*}
		which implies $\dim(W\cap W_{1})\leq t-1$.	Since $\mf$ is $s$-almost $t$-intersecting, we know $\mw\nsubseteq\mf$ or $W_{1}\notin \mf$. 
		On the other hand, since $\mf$ is maximal, we have
		$\mw\cup\lbrace W_{1}\rbrace\subseteq \mf$, a contradiction.  This finishes the proof.
	\end{proof}
	
	\begin{lem}\label{2409121}
		Let $n$, $k$ and $t$ be positive integers with $k\geq t+1$ and $n\geq 2k-t+1$. Assume that $\mt\subseteq {V\b t+1}$ is $t$-intersecting and its size is at least 2. Suppose that  $A$ and $B$ are $k$-subspaces of $V$ with $\dim(A\cap B)<t$. If $\dim(A\cap T)\geq t$ and $\dim(B\cap T)\geq t$ for any $T\in \mt$, then $\bigcap_{T\in \mt}T\in {A\b t}\cup {B\b t}$.
	\end{lem}
	\begin{proof}
		Since $\left|\mt\right|\geq 2$, we have $\dim(\bigcap_{T\in \mt}T)\leq t$. It is sufficient to prove that  $\bigcap_{T\in \mt}T$ contains a $t$-subspace in ${A\b t}\cup {B\b t}$. 
		
		Choose $T_{1},T_{2}\in\mt$. We  first show  
		\begin{equation}\label{add1}
			T_{1}\cap A=T_{2}\cap A\quad \textnormal{or}\quad T_{1}\cap B=T_{2}\cap B.
		\end{equation}
		Suppose for contradiction that $T_{1}\cap A\neq T_{2}\cap A$ and $T_{1}\cap B\neq T_{2}\cap B$. Recall that $\mt$ is $t$-intersecting. Then $\dim(T_{1}+T_{2})=t+2$. Since  $\dim(A\cap T)\geq t$ and $\dim(B\cap T)\geq t$ for any $T\in \mt$, we have 
		\begin{equation*}
			\dim(T_{1}\cap A+T_{2}\cap A)\geq t+1\quad \textnormal{and}\quad\dim(T_{1}\cap B+T_{2}\cap B)\geq t+1.
		\end{equation*}
		Therefore, we obtain 
		\begin{equation*}
			\begin{aligned}
				t+2=&\dim(T_{1}+T_{2}) \\
				\geq & \dim((T_{1}+T_{2})\cap A+(T_{1}+T_{2})\cap B)\\
				\geq & \dim(T_{1}\cap A+T_{2}\cap A)+\dim(T_{1}\cap B+T_{2}\cap B)-\dim(A\cap B)\\
				\geq &\ t+3,
			\end{aligned}
		\end{equation*}
		a contradiction. Thus \eqref{add1} holds.
		
		Let $W_{1}$ and $W_{2}$ be distinct members of $\mt$. By \eqref{add1}, w.l.o.g., assume that $W_{1}\cap A=W_{2}\cap A:=E$. It follows that $\dim(E)=t$ from Lemma \ref{2405283}. Next we show $E=T\cap A$ for any $T\in \mt$. Otherwise, \eqref{add1} implies $T\cap B=W_{1}\cap B$ and $T\cap B=W_{2}\cap B$. From Lemma \ref{2405283}, we  know 
		$$W_{1}=W_{1}\cap A+W_{1}\cap B=W_{2}\cap A+W_{2}\cap B=W_{2},$$ a contradiction to the assumption that $W_{1}\neq W_{2}$. This yields $E\subseteq \bigcap_{T\in \mt}T$, as desired.
	\end{proof}
	
	\begin{pr1}\label{2407081}
		Let $n$, $k$, $t$ and $s$ be positive integers with $(t,s)\neq (1,1)$, $k\geq t+2$ and $n\geq 2k+t+s-1+\delta_{2,q}$. Suppose $\mf\subseteq {V\b k}$ is a maximal $s$-almost $t$-intersecting family which is not $t$-intersecting. If $\tau_{t}(\mf)\geq t+2$, then $\left|\mf\right|<h(n,k,t,s)$.
	\end{pr1}

	The proof of Proposition \ref{2407081} is tedious. Thus we only present this proposition here, and prove it in Section \ref{2407065}.

	\begin{proof}[\bf Proof of Theorem \ref{2407163}]
	Recall that the family defined in Construction \ref{2406191} is $s$-almost $t$-intersecting but not $t$-intersecting. Then
	\begin{equation}\label{2406152}
		\left|\mf\right|\geq h(n,k,t,s).
	\end{equation}
	Observe that there exist $A,B\in \mf$ such that
	$\dim(A\cap B)<t$. Hence $\tau_{t}(\mf)\geq t+1$. 
	
   	We say that $\mf$ is a maximal $s$-almost $t$-intersecting family. Otherwise, there exists $D\in{V\b k}\bs \mf$ such that $\mf\cup\left\{D\right\}$ is $s$-almost $t$-intersecting. Since $\left\{A, B\right\}\subseteq \mf\subseteq \mf\cup\left\{D\right\}$, we further derive that $\mf\cup\left\{D\right\}$ is an $s$-almost $t$-intersecting family which is not $t$-intersecting. This contradicts to the assumption that $\mf$ has maximum size. 	If $\tau_t(\mf)\geq t+2$, then by   Proposition \ref{2407081}, we have 
	$|\mf|<h(n,k,t,s)$. This contradicts (\ref{2406152}). Therefore $\tau_{t}(\mf)=t+1$.
	Let  $\mt$ denote the set of all $t$-covers of $\mf$  with dimension $t+1$.  
	\begin{cl}\label{2407174}
		$\left|\mt\right|\geq 2$.
	\end{cl}
	\noindent \textit{Proof.} 
	Suppose for contradiction that $\mt=\lbrace T \rbrace$. Then
	$\mf=\mf_{T}\cup \left( \bigcup_{H\in {T\b t}}\left(\mf \bs \mf_{T}\right)_{H} \right)$.
	From Lemma \ref{2407022}, we obtain 
	$$\left|\mf\right|\leq {n-t-1\b k-t-1}+{t+1\b 1}{k-t+1\b 1}^{2}{n-t-2\b k-t-2}+s{t+1\b 1}{k-t+1\b 1}+s{t+1\b 1}.$$
		By Lemma \ref{1}, we have 
	\begin{equation*}\label{2405303}
		\begin{aligned}
			\frac{\left|\mf\right|}{{k-t+1\b 1}{n-t-1\b k-t-1}}\leq &\  \frac{1}{{k-t+1\b 1}}+\frac{{t+1\b 1}{k-t+1\b 1}(q^{k-t-1}-1)}{q^{n-t-1}-1}+\frac{s{t+1\b1 }}{{n-t-1\b k-t-1}}+\frac{s{t+1\b 1}}{{k-t+1\b 1}{n-t-1\b k-t-1}}\\
			\leq &\  \frac{1}{13}+\frac{9}{4q^{n-2k}}+\frac{3s}{2q^{n-k-t}}+\frac{3s}{26q^{n-k-t}}\leq \frac{181}{468}<\frac{23}{24}
		\end{aligned}
	\end{equation*}
	for $q\geq 3$, and 
	\begin{equation*}\label{2405304}
		\begin{aligned}
			\frac{\left|\mf\right|}{{k-t+1\b 1}{n-t-1\b k-t-1}}\leq &\ \frac{1}{7}+\frac{(q^{t+1}-1)(q^{k-t+1}-1)(q^{k-t-1}-1)}{q^{n-t-1}-1}+\frac{s(q^{t+1}-1)}{{n-t-1\b k-t-1}}+\frac{s(q^{t+1}-1)}{7{n-t-1\b k-t-1}}\\
			\leq &\ \frac{1}{7}+\frac{1}{q^{n-2k-2}}+\frac{s}{q^{n-k-t-1}}+\frac{s}{7q^{n-k-t-1}}\leq \frac{11}{14}<\frac{23}{24} 
		\end{aligned}
	\end{equation*}
	for $q=2$. By Lemma \ref{2405262}, we derive $\left| \mf\right|<h(n,k,t,s)$, which contradicts (\ref{2406152}).
	$\hfill\square$

    Let $E=\bigcap_{T\in \mt}T$ and $M$ be the space spanned by the union of all members of $\mt$. 	Notice that $\mf$ is maximal. By Lemma \ref{2405122}, we know $\mt$ is $t$-intersecting. This combining with Lemma \ref{2409121} yields $E\in {A\b t}\cup {B\b t}$. W.l.o.g., assume that $E\in {A\b t}$. 
    
    \begin{cl}\label{2407061} 
    	The following hold.
    	\begin{enumerate}[\normalfont(i)]
    		\item  $M\subseteq E+F\in {V\b k+1}$ for each $F\in \mathcal{F}\backslash \mathcal{F}_{E}$.
    		\item  $A\cap M=E$.
    	\end{enumerate}
    \end{cl}	
	\noindent \textit{Proof.} 
	(i) For any $T\in \mathcal{T}$ and $F\in \mf\bs \mf_{E}$, since $E\subseteq T$ and $\dim(T\cap F)\geq t$, we have $\dim(E\cap F)=t-1$ and $\dim(T\cap (E+F))\geq t+1$, which   imply that $\dim(E+F)=k+1$ and $T\subseteq E+F$. Hence $M\subseteq E+F\in {V\brack k+1}$. 
	
	(ii) It is easy to check $E\subseteq A\cap M$. In order to prove $A\cap M=E$, we just need to show $\dim(A\cap M)=t$. Suppose for contradiction that $\dim(A\cap M)\geq t+1$. Observe that $B\in \mf \bs \mf_{E}$. From (i), we obtain $M\subseteq E+B$ and $\dim(E+B)=k+1$. Then 
	\begin{equation*}
		\begin{aligned}
			\dim(A\cap B)= &\dim(A\cap (E+B)\cap B)\\
			\geq & \dim(A\cap (E+B))+\dim(B)-\dim(E+B)\\
			\geq &\ t,
		\end{aligned}
	\end{equation*} 
	a contradiction. This completes the proof.
	$\hfill\square$
	
	By Claims \ref{2407174} and \ref{2407061}, we have $t+2\leq \dim(M)\leq k+1$. 
	Next  we show a precise result.
	
	\begin{cl}\label{2407175}
		$\dim(M)=k+1$.
	\end{cl}
	\noindent \textit{Proof.} 
	Suppose for contradiction that $t+2\leq \dim(M)\leq k$. Denote $z=\dim(M)$ and $$\mf^{\prime}=\left\{ F\in\mf:\dim(F\cap A)\geq t,\ \dim(F\cap B)\geq t\right\}.$$

	We first show an upper bound on $\left|\mf^{\prime}_{E}\right|$. For any $F\in \mathcal{F}^{\prime}_{E}$, we have $\dim(F\cap (E+B))\geq t+1$ by $E\subseteq F$ and $\dim(F\cap B)\geq t$. Hence 
	\begin{equation}\label{2404251}
		\mathcal{F}^{\prime}_{E}=\left( \bigcup _{H_{1}\in {M\brack t+1},\ E\subseteq H_{1}}\mathcal{F}^{\prime}_{H_{1}}\right)\cup \left( \bigcup_{H_{2}\in {E+B\brack t+1}\backslash {M\brack t+1},\ E\subseteq H_{2}}    \mathcal{F}^{\prime}_{H_{2}}  \right). 
	\end{equation}
By Lemma \ref{2404234}, we have
$\left| \left\{ H_{1}\in {M\brack t+1}: E\subseteq H_{1}\right\} \right|={z-t\b 1}$ and  $\left| \mathcal{F}^{\prime}_{H_{1}} \right|\leq {n-t-1\brack k-t-1}$ for any $H_{1}\in {M\brack t+1}$,  
which imply
\begin{equation}
	\label{2404252}	\left| \bigcup _{H_{1}\in {M\brack t+1},\ E\subseteq H_{1}}\mathcal{F}^{\prime}_{H_{1}}\right|\leq {z-t\brack 1}{n-t-1\brack k-t-1}.
\end{equation}
Pick $H_{2}\in {E+B\brack t+1}\backslash {M\brack t+1}$ with $E\subseteq H_{2}$. 
Since $H_{2}$ is not a $t$-cover of $\mf$, then   $\dim(X\cap H_{2})<t$ for some  $X\in \mf$. By Lemma \ref{2404224}, there exits a $(2t+1-\dim(X\cap H_{2}))$-subspace $R$ such that  $\left|\mf_{H_{2}}\right|\leq {k-t+1\b1}^{t-\dim(X\cap H_{2})}\left|\mf_{R}\right|+s$. From Lemma \ref{1} (v), we obtain $$\left| \mathcal{F}^{\prime}_{H_{2}}\right|\leq \left|  \mathcal{F}_{H_{2}}\right|\leq {k-t+1\brack 1}{n-t-2\brack k-t-2}+s.$$
It follows that $B\in \mf\bs \mf_{E}$ from $E\in {A\b t}$ and $\dim(A\cap B)<t$.  By Claim \ref{2407061} (i), we know $E\subseteq M\subseteq E+B\in {V\b k+1}$. This combining with Lemma \ref{2404234} yields 
$$\left| \left\{  H_{2}\in {E+B\brack t+1}\backslash {M\brack t+1}: E\subseteq H_{2} \right\}\right|={k-t+1\brack 1}-{z-t\brack 1}=q^{z-t}{k-z+1\brack 1}.$$
Therefore,  we have
\begin{equation}\label{2405311}
	\left|\bigcup_{ H_{2}\in {E+B\brack t+1}\backslash {M\brack t+1},\ E\subseteq H_{2}}    \mathcal{F}^{\prime}_{H_{2}}  \right|\leq q^{z-t}{k-z+1\brack 1}{k-t+1\brack 1}{n-t-2\brack k-t-2}+s q^{z-t}{k-z+1\brack 1}.
\end{equation}

	Next we give an upper bound on $\left| \mathcal{F}^{\prime}\backslash \mathcal{F}^{\prime}_{E}\right|$. 
	For any $F\in \mathcal{F}^{\prime}\backslash \mathcal{F}^{\prime}_{E}\subseteq \mathcal{F}\backslash \mathcal{F}_{E}$, by Claim \ref{2407061} (i), we have $E+F\subseteq M+F\subseteq E+F$ and $\dim(E+F)=k+1$, which imply $\dim(F\cap M)=z-1$. Note  that $A\cap M=E$ from Claim \ref{2407061} (ii). Then
	\begin{equation*}
		\begin{aligned}
			\dim(F\cap (A+M))\geq &\dim(F\cap A)+\dim(F\cap M)-\dim(F\cap A\cap M)\\
			\geq&\ t+z-1-\dim(F\cap E)\\
			\geq& \ z.
		\end{aligned}
	\end{equation*}
	Hence, we know 
	\begin{equation*}
		\mathcal{F}^{\prime}\backslash \mathcal{F}^{\prime}_{E}\subseteq \bigcup_{H_{3}\in {A+M\brack z},\ E \nsubseteq H_{3},\ \dim(H_{3}\cap M)=z-1} \mathcal{F}^{\prime}_{H_{3}}.
	\end{equation*}
	It follows 	from Lemma \ref{2404234} that 
	$$\left| \left\{  H_{3}\in {A+M\brack z}:  E \nsubseteq H_{3},\ \dim(H_{3}\cap M)=z-1 \right\}\right|=q^{z -t+1}{t\brack 1}{k-t\brack 1}.$$
	Therefore, we derive
	\begin{equation}
		\label{2404254}
		\left| \mathcal{F}^{\prime}\backslash \mathcal{F}^{\prime}_{E}\right|\leq q^{z -t+1}{t\brack 1}{k-t\brack 1}{n-z \brack k-z}.
	\end{equation}

	Note that $\mathcal{F}=\mathcal{F}^{\prime}_{E}\cup (\mathcal{F}^{\prime}\backslash \mathcal{F}^{\prime}_{E})\cup \md_{\mf}(A;t)\cup \md_{\mf}(B;t)$. Then $\left|\mf\right|\leq  g(n,k,t,s, z)$ from (\ref{2404251})--(\ref{2404254}), where $g(n,k,t,s,x)$ is defined in (\ref{2407186}). By Lemma \ref{2407176}, we further conclude 
	\begin{equation}\label{2407177}
		\left|\mf\right|\leq g(n,k,t,s,k).
	\end{equation}	
	From Lemma \ref{1},  we obtain 
	\begin{equation*}
		\begin{aligned}
			\frac{g(n,k,t,s, k)}{{k-t+1\b 1}{n-t-1\b k-t-1}}
			=&\ \frac{{k-t\b 1}}{{k-t+1\b 1}}+\frac{q^{k-t}(q^{k-t-1}-1)}{q^{n-t-1}-1}+\frac{s q^{k-t}+q^{k-t+1}{t\b 1}{k-t\b 1}+2s}{{k-t+1\b 1}{n-t-1\b k-t-1}}\\
			\leq&\ \frac{1}{q}+\frac{1}{q^{n-2k+t}}+\frac{s}{q^{n-k}}+\frac{1}{q^{n-2k-1}}+\frac{2s}{q^{n-t}}\\
			\leq&\ \max\left\{ \frac{524}{729}, \frac{55}{64}\right\} <\frac{23}{24}.
		\end{aligned}
	\end{equation*}
	This together with  Lemma \ref{2405262} and (\ref{2407177}) yields $\left|\mf\right|<h(n,k,t)$, a contradiction to (\ref{2406152}). The desired result holds. $\hfill\square$

For convenience, write 	
$$\mh(M,E)=\left\{F\in{V\b k}: E\subseteq F,\ \dim(F\cap M)\geq t+1\right\}.$$
Let $\mf=\mf_{1}\sqcup \mf_{2} \sqcup \left(\mf\bs \mf_{E}\right)$, where
$$\mf_{1}=\left\{F\in \mf: E\subseteq F,\ \dim(F\cap B)\geq t \right\},$$ $$\mf_{2}=\left\{F\in \mf: E\subseteq F,\ \dim(F\cap B)<t\right\}.$$

Recall that $E\in {A\b t}$ and $\dim(A\cap B)<t$. Then $B\in \mf\bs \mf_{E}$. By Claim \ref{2407061} (i) and Claim \ref{2407175}, we have $M=E+B$. Hence
$$\dim(F_{1}\cap M)=\dim(F_{1}\cap (E+B))\geq t+1$$ 
for any $F_{1}\in\mf_{1}$, which implies $\mf_{1}\subseteq \mh(M,E)$. 

Since $M=E+B$, for any $F_{2}\in \mf_{2}$, we have 
\begin{equation*}
	\begin{aligned}
		\dim(F_{2}\cap M)=&\dim(F_{2})+\dim(M)-\dim(F_{2}+M)\\
		             =&\ 2k+1-\dim(F_{2}+B)\\
		             \leq &\ t,		
	\end{aligned}
\end{equation*}
implying  $F_{2}\cap M=E$. Hence $\mf_{2}\subseteq \mh_{1}(M,E)$. Moreover, we have $\left|\mf_{2}\right|\leq \left|\md_{\mf}(B;t)\right|\leq s$.

Choose $F\in \mf\bs\mf_{E}$, by Claim \ref{2407061} (i) and Claim \ref{2407175}, we know $F\subseteq E+F=M$. Hence $\mf\bs \mf_{E}\subseteq \mh_{2}(M,E)$.   Note that $\dim(F\cap A)=\dim(F\cap M\cap A)=\dim(F\cap E)<t$ due to Claim \ref{2407061} (ii). Then 
\begin{equation}\label{2409141}
	\left|\mf\bs \mf_{E}\right|\leq \left|\md_{\mf}(A;t)\right|\leq s.
\end{equation}Since $\dim(M)=k+1$ and $E\subseteq M$, we get 
$$\mh_{2}(M,E)=\left\{F\in {M\b k}: \dim(F\cap E)=t-1 \right\}.$$
It follows that $\left|\mh_{2}(M,E)\right|=q^{k-t+1}{t\b 1}$ from Lemma \ref{2404234}. This together with (\ref{2409141}) yields $\left|\mf\bs \mf_{E}\right|\leq \min\left\{ s, q^{k-t+1}{t\b 1}\right\}$.

The discussion above combining with (\ref{2406152}) produces 
$$h(n,k,t,s)\leq \left|\mf\right|=\left|\mf_{1}\right|+\left|\mf_{2}\right|+\left|\mf\bs \mf_{E}\right|\leq \left|\mh(M,E)\right|+s +\min\left\{ s, q^{k-t+1}{t\b 1}\right\}=h(n,k,t,s).$$
Hence $\mf_{1}=\mh(M,E)$, $\mf_{2}$ is an $s$-subset of $\mh_{1}(M,E)$ and $\mf\bs \mf_{E}$ is a $\min\left\{ s, q^{k-t+1}{t\b 1}\right\}$-subset of $\mh_{2}(M,E)$. The desired result holds.
	\end{proof}

\section{Proof of Proposition \ref{2407081}}\label{2407065}

To prove Proposition \ref{2407081}, we need the following lemmas.	
	
	\begin{lem}\label{2405232}
	Let $U_{1}, U_{2}, U_{3}, U_{4}$ be subspaces of $V$. If $U_{1}\cap U_{2}\subseteq U_{j}$ and $U_{j}=U_{j}\cap U_{1}+U_{j}\cap U_{2}$ for any $j\in \lbrace 3,4 \rbrace$, then 
	$$U_{3}\cap U_{4}=U_{3}\cap U_{4}\cap U_{1}+U_{3}\cap U_{4}\cap U_{2}.$$
\end{lem}
\begin{proof}
	Choose $u\in U_{3}\cap U_{4}$. 
	Then there exist $u_{3,1}\in U_3\cap U_1$,  $u_{3,2}\in U_3\cap U_2$, $u_{4,1}\in U_4\cap U_1$ and $u_{4,2}\in U_4\cap U_2$ such that $$u=u_{3,1}+u_{3,2}=u_{4,1}+u_{4,2}.$$
	Since $u_{3,1}-u_{4,1}=u_{4,2}-u_{3,2}$, we know $u_{3,1}-u_{4,1}\in U_{1}\cap U_{2}$. 
	This together with $U_{1}\cap U_{2}\subseteq U_{4}\cap U_{1}$ yields $u_{3,1}\in U_{4}\cap U_{1}$. 
	Hence $u_{3,1}\in U_{3}\cap U_{4}\cap U_{1}$. Similarly, we have $u_{3,2}\in U_{3}\cap U_{4}\cap U_{2}$. 
	Therefore, we derive
	$$u\in U_{3}\cap U_{4}\cap U_{1}+U_{3}\cap U_{4}\cap U_{2},$$ 
	implying that LHS $\subseteq$ RHS. It is clear that  LHS $\supseteq$ RHS. 
	The desired result follows.
\end{proof}	
	
	Let $X_{1}$ and $X_{2}$ be subspaces of $V$ and $\mf\subseteq {V\b k}$. Write
	$$\mf(X_{1},X_{2}; t, i)=\left\{F\in \mf: \dim(F\cap X_{1})\geq t,\ \dim(F\cap X_{2})\geq t,\ \dim(F\cap X_{1}\cap  X_{2})=i\right\}.$$ 
	
\begin{lem}\label{2405281}
	Let $n$, $k$, $t$, $s$, $\ell$  be positive integers and $i$, $j$ be non-negative integers with $i\leq \min\lbrace t,j \rbrace$, $2t-i\leq k$, $t+j-i\leq \ell$ and $n\geq 2k$. Suppose that $\mf\subseteq {V\b k}$ is an $s$-almost $t$-intersecting family with $2t-i\leq \tau_{t}(\mf)\leq k$. 
	If $X_{1}, X_{2}\in{V\b\ell}$ with $\dim(X_{1}\cap X_{2})=j$, then
	\begin{equation*}
		\begin{aligned}
			\frac{\left| \mf(X_{1},X_{2};t, i)\right|}{q^{2(j-i)(t-i)}{j\b i}{\ell-j\b t-i}^{2}}	
			\leq {k-t+1\brack 1}^{\tau_{t}(\mathcal{F})+i-2t}{n-\tau_{t}(\mathcal{F})\brack k-\tau_{t}(\mathcal{F})}
			+s\sum_{m=0}^{\tau_{t}(\mathcal{F})+i-2t-1}{k-t+1\brack 1}^{m}.
		\end{aligned}
	\end{equation*} 
\end{lem}
\begin{proof}
	Let $F\in \mf(X_{1}, X_{2}; t,i)$. By the definition, there exist $Z_{1}\in {F\cap X_{1}\b t}$ and $Z_{2}\in {F\cap X_{2}\b t}$ such that $Z_{1}\cap X_{2}=Z_{2}\cap X_{1}=F\cap X_{1}\cap X_{2}$. Observe that  $Z_{1}\cap Z_{2}=F\cap X_{1}\cap X_{2}$. 
	This together with $\dim(X_{1}+Z_{2})=\ell+t-i$ and $\dim(X_{2}+Z_{1})=\ell+t-i$ yields $$\dim((Z_{1}+Z_{2})\cap X_{1})=\dim((Z_{1}+Z_{2})\cap X_{2})=t.$$ We further get  $(Z_{1}+Z_{2})\cap X_{1}\cap X_{2}=F\cap X_{1}\cap X_{2}$. Then
	\begin{equation}\label{24071713}
		\mf(X_{1},X_{2};t,i)\subseteq \bigcup_{Y\in {X_{1}\cap X_{2}\b i}} \ \bigcup_{Z\in \ma(Y)}\mf_{Z},
	\end{equation}
	where $$\ma(Y)=\left\{Z\in {V\b 2t-i}: \dim(Z\cap X_{1})=\dim(Z\cap X_{2})=t,\ Z\cap X_{1}\cap X_{2}=Y\right\}.$$
	
	Let $Y\in {X_{1}\cap X_{2}\b i}$ and $Z\in\ma(Y)$. We claim that 
	\begin{equation}\label{2407172}
		\left|\mf_{Z}\right|\leq {k-t+1\brack 1}^{\tau_{t}(\mathcal{F})+i-2t}{n-\tau_{t}(\mathcal{F})\brack k-\tau_{t}(\mathcal{F})}
		+s\sum_{m=0}^{\tau_{t}(\mathcal{F})+i-2t-1}{k-t+1\brack 1}^{m}
	\end{equation}
	Recall that $\tau_{t}(\mf)\geq 2t-i$. If $\tau_{t}(\mf)=2t-i$, then (\ref{2407172}) follows from $\left|\mf_{Z}\right|\leq {n-(2t-i)\b k-(2t-i)}$. If $\tau_{t}(\mf)>2t-i$, then (\ref{2407172}) holds from Corollary \ref{2407166}. 
	
	Note that $Z= Z\cap X_{1}+Z\cap X_{2}$ since $\dim(Z)=\dim(Z\cap X_{1}+Z\cap X_{2})$. By Lemma \ref{2404234}, we have 
	\begin{equation}\label{2407173}
		\left|\ma(Y)\right|\leq \left| \left\{Z_{1}\in {X_{1}\b t}: Z_{1}\cap X_{1}\cap X_{2}=Y\right\}\right|^{2}=q^{2(j-i)(t-i)}{\ell-j\b t-i}^{2}.
	\end{equation}
	
	From (\ref{24071713})--(\ref{2407173}), we get the desired result.
\end{proof}

\begin{proof}[\bf Proof of Proposition \ref{2407081}]
	By Lemma \ref{2405262}, it is sufficient to show 
	\begin{equation*}
		\left|\mf\right|<\left\{
		\begin{aligned}
			&\frac{71}{72}{k-t+1\b1}{n-t-1\b k-t-1} &\textnormal{if}\ \  q\geq 3,\\
			&\frac{23}{24}{k-t+1\b1}{n-t-1\b k-t-1} &\textnormal{if}\ \  q=2 .
		\end{aligned}
		  \right.
	\end{equation*}
	We investigate $\mf$  in three cases. 
	For convenience,  denote the set of all $t$-covers of $\mf$  with dimension $\tau_{t}(\mf)$ by $\mt$.

	\medskip
	\noindent{{\bf Case 1.} $\tau_t(\mf)=t+2<k$.}
	\medskip

	\noindent{{\bf Case 1.1.} $\left| \mt\right|=1$.}
	\medskip
	
	Set $\mt=\lbrace T \rbrace$. Then 
	$\mf=\mf_{T}\cup \left( \bigcup_{H\in {T\b t}}\left(\mf \bs \mf_{T}\right)_{H} \right)$.
	It follows from Lemma \ref{2407022} that
	$$\left|\mf\right|\leq {n-t-2\b k-t-2}+{t+2\b 2}\left({k-t+1\b 1}^{3}{n-t-3\b k-t-3} + s{k-t+1\b 1}^{2}+s{k-t+1\b 1}+s\right).$$
	By Lemmas \ref{1}  and \ref{2405261}, we obtain
	\begin{equation*}
		\begin{aligned}
			\frac{\left|\mf\right|}{{k-t+1\b 1}{n-t-1\b k-t-1}}\leq &\  \frac{1}{q^{n-t}}+\frac{7}{2q^{2n-4k-2}}+\frac{7s}{2q^{2n-3k-t-1}}+\frac{7s}{2q^{2n-2k-2t}}+\frac{7s}{2q^{2n-k-3t}}.
		\end{aligned}
	\end{equation*}	
 Recall that $n\geq 2k+t+s-1+\delta_{2,q}$ and $t+s\geq 3$. If $q\geq 3$, then 
 $$\frac{\left|\mf\right|}{{k-t+1\b 1}{n-t-1\b k-t-1}}\leq \frac{1}{q^{2k+s-1}}+\frac{7}{2q^{2t+2s-4}}+\frac{7s}{2q^{k+t+2s-3}}+\frac{7s}{2q^{2k+2s-2}}+\frac{7s}{2q^{3k+2s-t-2}}\leq \frac{71567}{177147}<\frac{71}{72}.$$
 If $q= 2$, then  
 $$\frac{\left|\mf\right|}{{k-t+1\b 1}{n-t-1\b k-t-1}}\leq \frac{1}{2^{2k+s}}+\frac{7}{2^{2t+2s-1}}+\frac{7s}{2^{k+t+2s}}+\frac{7s}{2^{2k+2s+1}}+\frac{7s}{2^{3k+2s-t+1}}\leq \frac{4127}{16384}<\frac{23}{24}.$$
 The desired result follows.
 
 \medskip
 \noindent{{\bf Case 1.2.} $\left| \mt\right|\geq 2$.}
 \medskip
	
	Let $T_{1}$ and $T_{2}$  be distinct members of $\mt$. Then $\dim(T_{1}\cap T_{2})\in \lbrace t,t+1 \rbrace$ by Lemma \ref{2405122}.

	\medskip
	\noindent{{\bf Case 1.2.1.}  $\dim(T_{1}\cap T_{2})=t$.}
	\medskip

	For each $F\in \mf$, we have 
	$$t+2=\dim(T_{1})\geq \dim(F\cap T_{1}+T_{1}\cap T_{2})\geq 2t-\dim(F\cap T_{1}\cap T_{2}),$$
	implying that $\dim(F\cap T_{1}\cap T_{2})\geq t-2$. Hence
	$$\mf=\mf(T_{1}, T_{2};t,t)\cup \mf(T_{1}, T_{2};t, t-1)\cup \mf(T_{1}, T_{2};t, t-2).$$ 
	By Lemma \ref{2405281}, we know
	\begin{equation*}
		\begin{aligned}
			\left| \mf\right|\leq&\ \left({k-t+1\b1}^{2}{n-t-2\b k-t-2}+s{k-t+1\b 1}+s\right)+\left(q^{2}{t\b 1}{2\b 1}^{2}{k-t+1\b1}{n-t-2\b k-t-2}\right.\\
			&\ \left.+q^{2}s{t\b 1}{2\b 1}^{2}\right)+q^{8}{t\b 2}{n-t-2\b k-t-2}.		
		\end{aligned}
	\end{equation*}
  If $q\geq 3$, then $q+1\leq \frac{4}{3}q$. By  Lemmas \ref{1} and \ref{2405261}, we have 
	\begin{equation*}
		\begin{aligned}
			\frac{\left| \mf\right|}{{k-t+1\b 1}{n-t-1\b k-t-1}}\leq& \frac{3}{2q^{2t+s-1}}+\frac{s}{q^{2k+2t+2s-2}}+\frac{s}{q^{3k+t+2s-2}}+\frac{8}{3q^{(k-t)+(t+s)-4}}+\frac{8s}{3q^{3k+2s-5}}\\
			&+\frac{2}{q^{2k-2t+s-5}}\leq \frac{5492909}{9565938}<\frac{71}{72}. 
		\end{aligned}
	\end{equation*}
	If $q=2$ and $k=t+3$, then by  Lemmas \ref{1} and \ref{2405261}, we get 
	\begin{equation*}
		\begin{aligned}
			\frac{\left| \mf\right|}{{k-t+1\b 1}{n-t-1\b k-t-1}}\leq& \frac{45}{2^{2t+s+5}-1}+\frac{s}{2^{4t+2s+6}}+\frac{s}{15\cdot2^{4t+2s+6}}+\frac{108}{2^{t+s+5}}+\frac{36s}{15\cdot2^{3t+2s+6}}\\
			&+\frac{56}{15\cdot2^{s+3}}\leq \frac{1167905}{1569792}<\frac{23}{24}. 
		\end{aligned}
	\end{equation*}
   Notice that $\frac{2^{t}-1}{2^{2t}}\leq \frac{1}{4}$.	If $q=2$ and $k\geq t+4$, then by  Lemmas \ref{1} and \ref{2405261}, we derive 
	\begin{equation*}
		\begin{aligned}
			\frac{\left| \mf\right|}{{k-t+1\b 1}{n-t-1\b k-t-1}}\leq& \frac{1}{2^{2t+s-1}}+\frac{s}{2^{3k+3t+3s}}+\frac{s}{2^{4k+2t+3s}}+\frac{36}{2^{k-t+s}}\cdot\frac{2^{t}-1}{2^{2t}}+\frac{36s}{2^{4k+3s}}\cdot\frac{2^{t}-1}{2^{2t}}\\
			&+\frac{7}{2^{2k+s-2t-3}}\leq \frac{69206177}{134217728}<\frac{23}{24}. 
		\end{aligned}
	\end{equation*}
	These imply the desired result.

	\medskip
	\noindent{{\bf Case 1.2.2.}  $\dim(T_{1}\cap T_{2})=t+1$.}
	\medskip

	For each $F\in \mf$, we have 
	$$t+2=\dim(T_{1})\geq \dim(F\cap T_{1}+T_{1}\cap T_{2})\geq 2t+1-\dim(F\cap T_{1}\cap T_{2}),$$
	implying $\dim(F\cap T_{1}\cap T_{2})\geq t-1$. Hence
	$$\mf = \mf(T_{1},T_{2};t,t-1)\cup \left(\bigcup_{H\in{T_{1}\cap T_{2}\b t}}\mf_{H}\right).$$
	It follows from Corollary \ref{2407166}  and Lemma \ref{2405281} that 
	\begin{equation*}
		\begin{aligned}
			\left|\mf\right|\leq&\ q^{4}{t+1\b 2}{k-t+1\b 1}{n-t-2\b k-t-2}+sq^{4}{t+1\b 2}+ {t+1\b 1}{k-t+1\b 1}^{2}{n-t-2\b k-t-2}\\
			&\ +s{t+1\b 1}{k-t+1\b1 }+s{t+1\b 1}.
		\end{aligned}
	\end{equation*}
	If $q\geq 3$, then by  Lemmas \ref{1} and \ref{2405261}, we have 
	\begin{equation*}
		\begin{aligned}
			\frac{\left| \mf\right|}{{k-t+1\b 1}{n-t-1\b k-t-1}}\leq&\  \frac{2}{q^{k+s-t-3}}+\frac{2s}{q^{3k+2s-t-4}}+\frac{9}{4q^{t+s-1}}+\frac{3s}{2q^{2k+t+2s-2}}+\frac{3s}{2q^{3k+2s-2}}\\
			\leq&\ \frac{1948745}{2125764}<\frac{71}{72}. 
		\end{aligned}
	\end{equation*}
	If $q=2$ and $k=t+3$, then by  Lemma \ref{1}, we know 
	\begin{equation*}
		\begin{aligned}
			\frac{\left| \mf\right|}{{k-t+1\b 1}{n-t-1\b k-t-1}}\leq&\  
			 \frac{16(2^{t+1}-1)(2^{t}-1)}{2^{n-t-1}-1}+\frac{16s(2^{t+1}-1)(2^{t}-1)}{45\cdot2^{2n-2k}}+\frac{45(2^{t+1}-1)}{2^{n-t-1}-1}+\frac{s(2^{t+1}-1)}{2^{2n-2k}}\\
		     &+\frac{s(2^{t+1}-1)}{15\cdot2^{2n-2k}}\\
		     \leq&\ \frac{16}{2^{s+4}}+\frac{16s}{45\cdot2^{2t+2s+5}}+\frac{45}{2^{t+s+4}}+\frac{s}{2^{3t+2s+5}}+\frac{s}{15\cdot 2^{3t+2s+5}}<\frac{491}{576}<\frac{23}{24}. 
		\end{aligned}
	\end{equation*}
	If $q=2$ and $k\geq t+4$, then by  Lemmas \ref{1} and \ref{2405261}, we get
	\begin{equation*}
		\begin{aligned}
			\frac{\left| \mf\right|}{{k-t+1\b 1}{n-t-1\b k-t-1}}\leq&\  
		 \frac{7}{2^{n-k-2t-1}}+\frac{7s}{2^{3n-2k-3t-1}}+\frac{1}{2^{n-2k-2}}+\frac{s}{2^{3n-3k-t-1}}+\frac{s}{2^{3n-2k-2t-1}}\\
		 \leq &\ \frac{7}{2^{k+s-t-1}}+\frac{7s}{2^{2k+2(t+s)+s+7}}+\frac{1}{2^{t+s-2}}+\frac{s}{2^{3k+2(t+s)+s-1}}+\frac{s}{2^{3k+2(t+s)+s+3}}\\
		 \leq &\ \frac{31457311}{33554432}<\frac{23}{24}. 
		\end{aligned}
	\end{equation*}
	The desired result follows.

	\medskip
	\noindent{{\bf Case 2.} $\tau_t(\mf)=t+2=k$.}
	\medskip

	Let $U_1,U_2\in\mf$ with $\dim(U_1\cap U_2)\leq t-1$. 
	Suppose  that $\dim(U_{1}\cap U_{2})\leq  t-3$. Then for any $t$-cover $W$ of $\mf$, we have
	$$\dim(W)\geq \dim(W\cap U_{1}+W\cap U_{2})\geq \dim(W\cap U_{1})+\dim(W\cap U_{2})-\dim(U_{1}\cap U_{2})\geq t+3,$$
	a contradiction to the assumption that $\tau_{t}(\mf)=t+2$. Hence $\dim(U_{1}\cap U_{2})\geq t-2$.  Define $$\mf^{\prime}=\left\{F\in\mf: \dim(F\cap U_{1})\geq t,\ \dim(F\cap U_{2})\geq t\right\}.$$

	\medskip
	\noindent{{\bf Case 2.1.} $\dim(U_{1}\cap U_{2})=t-2$.}
	\medskip
	
In this case, we have $t\geq 2$.	
Let $U\in \mf^{\prime}\cup \mt$. Then $\dim(U\cap U_{1})\geq t$ and $\dim(U\cap U_{2})\geq t$.  
By Lemma \ref{2405283}, we have
\begin{equation}\label{2406062}
	\dim(U\cap U_{1})=\dim(U\cap U_{2})=t,\ U_{1}\cap U_{2}\subseteq U,\ U=U\cap U_{1}+U\cap U_{2}.
\end{equation}
Choose $T\in \mt$. We get 
$$\mf^{\prime}=\mf^{\prime}_{T\cap U_{1}}\cup \mf^{\prime}_{T\cap U_{2}}\cup \left(\mf^{\prime}\bs \left(\mf^{\prime}_{T\cap U_{1}}\cup \mf^{\prime}_{T\cap U_{2}}\right)\right).$$

For each $i\in\lbrace1,2 \rbrace$ and $F\in \mf^{\prime}_{T\cap U_{i}}$, it follows from (\ref{2406062}) that
$$\dim(F\cap (T\cap U_{i}+U_{3-i}))\geq \dim(F\cap T \cap U_{i}+F\cap U_{3-i})\geq t+2,$$
which implies $F\subseteq T\cap U_{i}+U_{3-i}$. Notice that $\dim(T\cap U_{i})=t$ and $\dim(T\cap U_{i}+U_{3-i})=t+4$ due to (\ref{2406062}). Then 
\begin{equation}\label{2406063}
	\left| \mf^{\prime}_{T\cap U_{i}} \right|\leq \left| \left\{ F\in {T\cap U_{i}+U_{3-i}\b t+2}: T\cap U_{i}\subseteq F\right\}\right|= {4\b 2}.
\end{equation}

		Let $F\in\mf^{\prime}\bs \left(\mf^{\prime}_{T\cap U_{1}}\cup \mf^{\prime}_{T\cap U_{2}}\right)$. We have 
	$$\dim(F\cap T\cap U_{1})\leq t-1,\quad\dim(F\cap T\cap U_{2})\leq t-1.$$
	By  Lemma \ref{2405232}, we get $F\cap T=F\cap T\cap U_{1}+F\cap T\cap U_{2}$. We also know $U_{1}\cap U_{2}\subseteq F\cap T$ by (\ref{2406062}). 
	These yield
	\begin{equation*}
		\begin{aligned}
			t\leq \dim(F\cap T) 		=\dim(F\cap T\cap U_{1})+\dim(F\cap T\cap U_{2})-\dim(U_{1}\cap U_{2})\leq t, 		
		\end{aligned}
	\end{equation*}
	implying that $\dim(F\cap T\cap U_{1})=\dim(F\cap T\cap U_{2})=t-1$.
	Notice that $F=F\cap U_{1}+F\cap U_{2}$ by (\ref{2406062}). Then 
	$$\mf^{\prime}\bs \left(\mf^{\prime}_{T\cap U_{1}}\cup \mf^{\prime}_{T\cap U_{2}}\right)\subseteq \left\{ H_{1}+H_{2}: H_{1}\in\mh_{1},\ H_{2}\in \mh_{2}\right\},$$
	where for each $i\in\lbrace1,2 \rbrace$, $$\mh_{i}=\left\{ H_{i}\in {U_{i}\b t}: U_{1}\cap U_{2}\subseteq H_{i},\ \dim(H_{i}\cap (T\cap U_{i}))=t-1\right\}.$$
	It follows from  Lemma \ref{2404234} that $\left| \mh_{1}\right|=\left|\mh_{2}\right|=q{2\b1}^{2}$. Hence
	\begin{equation}\label{2406065}
		\left|\mf^{\prime}\bs \left(\mf^{\prime}_{T\cap U_{1}}\cup \mf^{\prime}_{T\cap U_{2}}\right)\right|\leq \left|\mh_{1}\right|^{2}= q^{2}{2\b 1}^{4}.
	\end{equation}
	
	Note that  $n\geq 2k+t+s-1$ and $q+1\leq \frac{3}{2}q$ for  $q\geq 2$.	By  (\ref{2406063}), (\ref{2406065}) and  Lemma \ref{1}, we have 
	\begin{equation*}
		\begin{aligned}
			\frac{\left|\mf\right|}{{3\b 1}{n-t-1\b 1}}\leq &\ \frac{2(q^{4}-1)}{(q+1)(q^{n-t-1}-1)}+\frac{81(q+1)^{4}(q-1)}{16(q^{2}+q+1)(q^{n-t-1}-1)}+\frac{2s(q-1)}{(q^{2}+q+1)(q^{n-t-1}-1)}\\
			\leq &\ \frac{2}{(q+1)q^{2t+s-2}}+\frac{81}{16(q^{2}+q+1)q^{2t+s-5}}+\frac{2s}{(q^{2}+q+1)q^{2t+s+1}}\\
			\leq&\ \frac{545}{672}<\frac{23}{24},
		\end{aligned}
	\end{equation*}
	as desired.

	\medskip
	\noindent{{\bf Case 2.2.} $\dim(U_{1}\cap U_{2})=t-1$.}
	\medskip
	
	For any $F\in \mf^{\prime}$, we have 
	$$t+2=\dim(F)\geq \dim(F\cap U_{1}+F\cap U_{2})\geq 2t-\dim(F\cap U_{1}\cap U_{2}),$$
	implying that $\dim(F\cap U_{1}\cap U_{2})\geq t-2$. Then 
	\begin{equation}\label{2406071}
		\mf=\mf(U_{1}, U_{2};t, t-1)\cup \mf(U_{1}, U_{2};t, t-2)\cup \md_{\mf}(U_{1};t)\cup \md_{\mf}(U_{2};t).
	\end{equation}
	From Lemma \ref{2405281}, we know
	\begin{equation}\label{2406072}
		\left| \mf(U_{1}, U_{2};t, t-1) \right|\leq {3\b 1}^{3}+s{3\b 1}^{2}.
	\end{equation}

	\medskip
	\noindent{{\bf Case 2.2.1.} $t\neq  2$.}
	\medskip
	
	Suppose $t=1$. Then $ \mf(U_{1}, U_{2};t, t-2)=\emptyset$. By (\ref{2406071}) and (\ref{2406072}), we have 
	\begin{equation*}
		\begin{aligned}
			\frac{\left|\mf\right|}{{3\b 1}{n-t-1\b 1}}\leq \frac{{3\b 1}^{2}}{{n-2\b 1}}+\frac{s{3\b 1}}{{n-2\b 1}}+\frac{2s}{{3\b 1}{n-t-2\b1}}.
		\end{aligned}
	\end{equation*}
	If $q\geq 3$, then by Lemma \ref{1}, we have 
	\begin{equation*}
		\begin{aligned}
			\frac{\left|\mf\right|}{{3\b 1}{n-t-1\b 1}}\leq \frac{9}{4q^{n-7}}+\frac{3s}{2q^{n-5}}+\frac{2s}{(q^{2}+q+1)q^{n-3}}\leq \frac{3869}{4212}<\frac{71}{72}. 
		\end{aligned}
	\end{equation*}
	If $q= 2$, then by Lemma \ref{1}, we get
	\begin{equation*}
		\begin{aligned}
			\frac{\left|\mf\right|}{{3\b 1}{n-t-1\b 1}}\leq \frac{49}{2^{n-2}-1}+\frac{7s}{2^{n-3}}+\frac{2s}{7\cdot2^{n-3}}\leq \frac{17453}{28448}<\frac{23}{24}. 
		\end{aligned}
	\end{equation*}
	
	Suppose $t\geq 3$. 	It follows from Lemma \ref{2405281} that
	\begin{equation}\label{2407188}
		\left|\mf(U_{1}, U_{2};t, t-2) \right|\leq q^{4}{t-1\b 1}{3\b 1}^{2}.
	\end{equation}
	Note that $n\geq 2k+t+s-1+\delta_{2,q}$ and ${3\b 1}\leq 2q^{2}$ for $q\geq 2$. These together with (\ref{2406071})--(\ref{2407188}) and  Lemma \ref{1} yield
	\begin{equation*}
		\begin{aligned}
			\frac{\left| \mf\right|}{{3\b 1}{n-t-1\b 1}}\leq&\ \frac{{3\b 1}^{2}}{{n-t-1\b 1}}+\frac{s{3\b 1}}{{n-t-1\b 1}}+\frac{q^{4}{t-1\b 1}{3\b 1}}{{n-t-1\b 1}}+\frac{2s}{{3\b1}{n-t-1\b1}}\\
			\leq&\ \frac{1}{q^{n-t-8}} +\frac{s}{q^{n-t-5}}+\frac{2q^{6}(q^{t-1}-1)}{q^{n-t-1}-1}+\frac{2s}{q^{n-t}}\\
			\leq&\ \max \left\{\frac{46172}{59049}, \frac{657}{1024}\right\}<\frac{23}{24}.
		\end{aligned}
	\end{equation*}
	The desired result holds.
	
	\medskip
	\noindent{{\bf Case 2.2.2.} $t=2$.}
	\medskip
	
	We claim that  
	\begin{equation}\label{2407178}
		\left| \mf(U_{1}, U_{2};2, 0)\right|\leq 2q^{2}{3\b 1}+2q^{3}{2\b 1}^{2}{3\b 1}+s.
	\end{equation}
	W.l.o.g., assume that $\mf(U_{1}, U_{2};2, 0)\neq \emptyset$. 	
	Let $U_{3}\in \mf(U_{1}, U_{2};2,0)$ and
	$$\mi=\left\{I\in\mf(U_{1}, U_{2};2, 0): \dim(I\cap U_{3})\geq 2 \right\}.$$ 
	To prove (\ref{2407178}), we first show 
	\begin{equation}\label{2407179}
		\begin{aligned}
			\mi \subseteq \mw_{1}\cup \mw_{2} \cup \mw_{1}^{\prime}\cup \mw_{2}^{\prime},
		\end{aligned}
	\end{equation}
	where for $i\in \lbrace 1,2 \rbrace$,
	$$\mw_{i}=\left\{ U_{3}\cap U_{i}+X: X\in{U_{3-i}\b 2},\ \dim(X\cap U_{1}\cap U_{2})=0\right\},$$
	$$\mw_{i}^{\prime}=\left\{ X+Y: X\in {U_{i}\b 2},\ Y\in {U_{3-i}\b 2},\ \dim(X\cap U_{3}\cap U_{i})=1,\ \dim(Y\cap U_{1}\cap U_{2})=0  \right\}.$$
	
	Let $U\in \mf(U_{1}, U_{2};2, 0)$. We have 
	$$4=\dim(U)\geq \dim(U\cap U_{1}+U\cap U_{2})=\dim(U\cap U_{1})+\dim(U\cap U_{2})\geq 4,$$
	which implies 
	\begin{equation}\label{2406075}
		\dim(U\cap U_{1})=\dim(U\cap U_{2})=2,\ U=U\cap U_{1}+U\cap U_{2}\subseteq U_{1}+U_{2}. 
	\end{equation}
	Let $I\in \mi$. 
	Suppose  that $\dim(I\cap U_{3}\cap U_{1})=\dim(I\cap U_{3}\cap U_{2})=0$. From (\ref{2406075}), we have 
	$$\dim(I\cap U_{1}+U_{3}\cap U_{1})=\dim(I\cap U_{2}+U_{3}\cap U_{2})=4,$$
	implying $I\cap U_{1}+U_{3}\cap U_{1}=U_{1}$ and $I\cap U_{2}+U_{3}\cap U_{2}=U_{2}$. These together with (\ref{2406075}) yield
	$$U_{1}+U_{2}=I\cap U_{1}+U_{3}\cap U_{1}+I\cap U_{2}+U_{3}\cap U_{2}\subseteq I+U_{3}\subseteq U_{1}+U_{2}.$$
	Hence $\dim(I\cap U_{3})=t-1$, a contradiction. So
	\begin{equation}\label{2406076}
		\dim(I\cap U_{3}\cap U_{1})\geq 1\ \ \textnormal{or} \ \dim(I\cap U_{3}\cap U_{2})\geq 1.
	\end{equation} 
	By (\ref{2406075}), we know $I=I_{1}+I_{2}$, where $I_{i}\in{U_{i}\b 2}$ with $\dim(I_{i}\cap U_{1}\cap U_{2})=0$. It follows from (\ref{2406076}) that $\dim(I_{1}\cap U_{3}\cap U_{1})\geq 1$ or $\dim(I_{2}\cap U_{3}\cap U_{1})\geq 1$. Then (\ref{2407179}) holds. 
	
	Note that $\left|\mf(U_{1}, U_{2};2, 0)\right|\leq \left|\mi\right|+s$ since $\mf$ is $s$-almost $t$-intersecting.  This together with (\ref{2407179}) and Lemma \ref{2404234} yields (\ref{2407178}).
	
	From (\ref{2406071}), (\ref{2406072}) and (\ref{2407178}), we have
	\begin{equation*}
		\begin{aligned}
			\left|\mf\right|\leq {3\b 1}^{3}+s{3\b 1}^{2}+2q^{2}{3\b1}+2q^{3}{2\b1}^{2}{3\b1}+3s.
		\end{aligned}
	\end{equation*}
Recall $n\geq 2k+t+s-1+\delta_{2,q}$. Then by Lemma \ref{1}, we get
\begin{equation*}
	\begin{aligned}
		\frac{\left|\mf\right|}{{3\b 1}{n-3\b1}}\leq&\ \frac{{3\b 1}^{2}}{{n-3\b1 }}+\frac{s{3\b 1}}{{n-3\b 1}}+\frac{2q^{2} }{{n-3\b 1}}+\frac{ 2q^{3}{2\b 1}^{2}}{{n-3\b 1}}+\frac{3s}{{3\b 1}{n-3\b 1}}\\
		\leq&\ s\left( \frac{{3\b 1}^{2}}{{n-3\b1 }}+\frac{{3\b 1}}{{n-3\b 1}}+\frac{2q^{2} }{{n-3\b 1}}+\frac{ 2q^{3}{2\b 1}^{2}}{{n-3\b 1}}\right)+\frac{3s}{(q^{2}+q+1)q^{n-4}}\\
		=&\ \frac{s}{q^{s-1}}\left(   \frac{2q^{6}+3q^{5}-q^{4}+2q^{3}-3q^{2}-q-2}{q^{n-s-2}-1}\right)+\frac{3s}{(q^{2}+q+1)q^{n-4}}\\
		\leq &\     \frac{2q^{6}+3q^{5}-q^{4}+2q^{3}-3q^{2}-q-2}{q^{7+\delta_{2,q}}-1}+\frac{3}{(q^{2}+q+1)q^{6+\delta_{2,q}}}.
	\end{aligned}
\end{equation*}
If $q\geq 3$, then   $\frac{2q^{6}+3q^{5}-q^{4}+2q^{3}-3q^{2}-q-2}{q^{7}-1}$ is decreasing. We further conclude 	$\frac{\left|\mf\right|}{{3\b 1}{n-3\b1}}\leq\frac{3362269}{3452787}<\frac{71}{72}$.  If $q=2$, then we have 
$\frac{\left|\mf\right|}{{3\b 1}{n-3\b1}}\leq\frac{187133}{228480}<\frac{23}{24}$. The desired result follows.

	\medskip
	\noindent{{\bf Case 3.} $\tau_t(\mf)\geq t+3$.}
	\medskip

	\medskip
	\noindent{{\bf Case 3.1.} $\tau_t(\mf)\leq k$.}
	\medskip

	Let $T\in\mt$. Then 
	$\mf=\bigcup_{H\in {T\b t}}\mf_{H}$.
	It follows from  Lemma \ref{5} and Corollary \ref{2407166} that 
	$$\frac{\left|\mf\right|}{{k-t+1\b 1}{n-t-1\b k-t-1}}\leq \frac{{t+3\b3}{k-t+1\b1}^{2} (q^{k-t-1}- 1)(q^{k-t-2}-1)}{(q^{n-t-1}-1)(q^{n-t-2}-1)}+ s\sum_{i=0}^{2}\frac{{t+3\b 3}{k-t+1\b 1}^{i}}{{k-t+1\b 1}{n-t-1\b k-t-1}}.$$
	By   Lemmas \ref{1} and \ref{2405261}, we have 
	\begin{equation*}\label{2406021}
		\begin{aligned}
			\frac{\left|\mf\right|}{{k-t+1\b 1}{n-t-1\b k-t-1}}&\leq 
			\frac{9}{2q^{2n-4k-t}}+\frac{3s}{q^{2n-3k-2t}}+\frac{2s}{q^{2n-2k-3t}}+\frac{2s}{q^{2n-k-4t}}\\
			&\leq \frac{9}{2q^{2}}+\frac{3s}{q^{4+s-1}}+\frac{2s}{q^{7+s-1}}+\frac{2s}{q^{10+s-1}}\leq \frac{63535}{118098}<\frac{71}{72}
		\end{aligned}
	\end{equation*}
	for $q\geq 3$, and 
	\begin{equation*}\label{2406022}
		\begin{aligned}
			\frac{\left|\mf\right|}{{k-t+1\b 1}{n-t-1\b k-t-1}}\leq&\  \frac{7(q^{k-t+1}-1)^{2}}{2q^{2n-2k-3t}}+\frac{7 sq^{k-t+1}}{2q^{2n-2k-3t}}+\frac{7s}{2q^{2n-2k-3t}}+\frac{7s}{2q^{2n-k-4t}}\\
			\leq &\ \frac{7(q^{k-t+1}-1)^{2}}{2q^{2k+s+1-t}}+\frac{7q^{k-t+1}}{2q^{2k+s+1-t}}+\frac{7}{2q^{2k+s+1-t}}+\frac{7}{q^{3k+s+1-2t}}\\
			=&\  \frac{7(q^{2k-2t+2}-q^{k-t+1}+2)}{2q^{2k+s+1-t}}+\frac{7}{2q^{3k+s+1-2t}}\\
			\leq&\  \frac{7}{2q^{s+t-1}}+\frac{7}{2q^{3k+s+1-2t}}\leq \frac{7175}{8192}<\frac{23}{24}
		\end{aligned}
	\end{equation*}
	for $q=2$. Then the desired result follows.
	
	\medskip
	\noindent{{\bf Case 3.2.} $\tau_t(\mf)\geq k+1$.}
	\medskip
	
		Note that ${k-t+1\b 1}\geq {3\b 1}\geq 7$ and $q^{z}\geq z+1$ for non-negative integer $z$. 
	From Lemma \ref{1}, we get
	\begin{equation*}
		\begin{aligned}
			\frac{s\binom{2k-2t+2}{k-t+1}}{{k-t+1\b 1}{n-t-1\b k-t-1}}&=\frac{2(2+\frac{1}{k-t})s\binom{2k-2t}{k-t-1}}{{k-t+1\b 1}{n-t-1\b k-t-1}}\leq \frac{5s\binom{2k-2t}{k-t-1}}{7q^{(n-k)(k-t-1)}}\leq \frac{5}{7}\cdot\frac{s(k-t+2)^{k-t-1}}{q^{\left(k+t+s-1\right)\left(k-t-1\right)}}\\
			&\leq \frac{5}{7}\cdot\frac{s}{q^{s-1}}\cdot\frac{(k-t+2)^{k-t-1}}{q^{(k+t)(k-t-1)}}\leq \frac{5}{7}\cdot\frac{s}{q^{s-1}}\cdot\frac{(k-t+2)^{k-t-1}}{(k+t+1)^{k-t-1}}\\
			&\leq \frac{5}{7}<\frac{23}{24}.
		\end{aligned}
	\end{equation*}
	This together with Lemma \ref{2407037} implies the desired result.
\end{proof}

	\medskip
	\noindent{\bf Acknowledgment.}	
	L. Ji is supported by the National Natural Science Foundation of China (No. 12271390). K. Wang is supported by the National Key R\&D Program of China (No. 2020YFA0712900) and National Natural Science Foundation of China (12071039, 12131011). T. Yao is supported by Guidance Plan for Key Scientific Research Projects of Higher Education Institutions in Henan Province (25B110003).

\end{document}